\documentclass[preprint,1p]{elsarticle}
\biboptions{sort&compress}

\usepackage[pdftex,colorlinks]{hyperref}
\usepackage[pagewise,mathlines]{lineno}
\newcommand*\patchAmsMathEnvironmentForLineno[1]{%
	\expandafter\let\csname old#1\expandafter\endcsname\csname #1\endcsname
	\expandafter\let\csname oldend#1\expandafter\endcsname\csname end#1\endcsname
	\renewenvironment{#1}%
	{\linenomath\csname old#1\endcsname}%
	{\csname oldend#1\endcsname\endlinenomath}}%
\newcommand*\patchBothAmsMathEnvironmentsForLineno[1]{%
	\patchAmsMathEnvironmentForLineno{#1}%
	\patchAmsMathEnvironmentForLineno{#1*}}%
\AtBeginDocument{%
	\patchBothAmsMathEnvironmentsForLineno{equation}%
	\patchBothAmsMathEnvironmentsForLineno{align}%
	\patchBothAmsMathEnvironmentsForLineno{flalign}%
	\patchBothAmsMathEnvironmentsForLineno{alignat}%
	\patchBothAmsMathEnvironmentsForLineno{gather}%
	\patchBothAmsMathEnvironmentsForLineno{multline}%
}

\journal{arxiv}

 \usepackage{tikz}
 \usetikzlibrary{shapes,decorations}
 \usepackage{environ}

\definecolor{mycolor}{RGB}{255,251,204}
\tikzstyle{mybox} = [draw=yellow, very thick,
rectangle, rounded corners, inner sep=5pt, inner ysep=10pt]
\tikzstyle{fancytitle} =[fill=white, text=red]

\NewEnviron{blc}{\begin{tikzpicture}
	\node [mybox] (box){%
		\begin{minipage}{0.95\textwidth}
		\BODY
		\end{minipage}
	};
	\end{tikzpicture}
}
\usepackage{amsfonts}
\usepackage{amssymb}
\usepackage[utf8]{inputenc}
\usepackage{amsmath}
\usepackage{amsthm}
\usepackage{graphicx}%
\usepackage{enumerate}
\usepackage{xcolor}
\usepackage{color,soul}
\usepackage{tikz}
\usetikzlibrary{patterns}
\usepackage{mathtools}

\newtheorem{thm}{Theorem}[section]
\newtheorem{cor}[thm]{Corollary}
\newtheorem{lem}[thm]{Lemma}
\newtheorem{prop}[thm]{Proposition}

\theoremstyle{definition}

\theoremstyle{remark}

\numberwithin{equation}{section}

\newcommand{\ep}{\varepsilon}
\newcommand{\R}{\mathbb{R}}				      % conjunto dos reais

\usepackage[textsize=scriptsize]{todonotes}
\usepackage{marginnote}

\allowdisplaybreaks[1] % permite quebra de páginas em equações

\begin{document}

\begin{frontmatter}

\title{\textsc{Carleman estimates for parabolic equations with super strong degeneracy in a set of positive measure}}

\author[UFCG]{Bruno S. V. Ara\'ujo}
\ead{bsergio@mat.ufcg.edu.br}

\author[RCN]{Reginaldo  Demarque\corref{mycorrespondingauthor}}
\cortext[mycorrespondingauthor]{Corresponding author}
\ead{reginaldo@id.uff.br}

\author[UEM]{Josiane C. O. Faria}
\ead{jcofaria@uem.br}

\author[GAN]{Luiz Viana}
\ead{luizviana@id.uff.br}

\address[UFCG]{Unidade Acadêmica de Matemática, Universidade Federal de Campina Grande, Campina Grande, PB, Brazil}
\address[RCN]{Departamento de Ciências da Natureza,
	Universidade Federal Fluminense,
	Rio das Ostras, RJ, Brazil}

\address[UEM]{Departamento de Matemática, Universidade Estadual de Maring\'a,
	Maring\'a, PR, Brazil}

 \address[GAN]{Departamento de Análise,
	Universidade Federal Fluminense,
	Niter\'{o}i, RJ, Brazil}
\begin{abstract}
	This work is concerned with the obtainment of new Carleman estimates for linear parabolic equations, where the second-order differential operator brings a super strong degeneracy in a positive measure subset of the spatial domain. In order to prove our main result, the control domain is supposed to contain the set of degeneracies. As a well-known consequence, we achieve a null controllability result in the current context. 
\end{abstract}


\end{frontmatter}

\section{Introduction}\label{intro}

In this paper, we study the null controllability of the following degenerate parabolic system

 \begin{equation}\label{prob0}
\begin{cases}
u_t-(a(x) u_x)_x +c(x,t)u=  f1_{\omega} &\text{in }    Q:=(0,1)\times(0,T),\\ u(0,t)=u(1,t)=0   & \text{in }  (0,T),\\ u(x,0)= u_0(x)  & \text{in } (0,1),
\end{cases}
\end{equation}
where $a\in W^{2,\infty}(0,1)$, $c\in L^\infty(Q)$, the control $f$ belongs to $L^2(Q)$,  $u_0\in L^2(0,1)$, the control domain $\omega\subset(0,1)$ is a non-empty open interval and $1_\omega$ denotes its associated characteristic function. 

We say that \eqref{prob0} is {\it null controllable} at time $T>0$ if, for any $u_0\in L^2(0,1)$, there exists a control function $f\in L^2(Q)$ such that the solution $u$ of \eqref{prob0} satisfies \begin{equation}
    \label{null}u(x,T)=0 \ \ \ \mbox{a. e. in} \ (0,1).
\end{equation}

The null controllability of degenerate parabolic equations, such as  \eqref{prob0}, has been extensively researched in the past two decades. One of the earliest works in this direction is \cite{cannarsa2006null}, due to Cannarsa and Fragnelli, where it is assumed that the function $a$ degenerates at the point $x=0$. To be more precise, the study in \cite{cannarsa2006null} considers two types of degeneracy, as described below. 

{\it \bf Weakly degenerate case (WDC):}

\begin{enumerate}
    \item $a\in\mathcal{C}([0,1])\cap\mathcal{C}^1((0,1]), \ a>0 \ \mbox{in} \ (0,1], \ a(0)=0$;
    \item $\exists K\in [0,1) \mbox{ such that } \ xa'(x)\leq Ka(x) \ \forall x\in[0,1]$.
\end{enumerate}

{\it \bf Strongly degenerate case (SDC):}

\begin{enumerate}
    \item $a\in\mathcal{C}^1([0,1]), \ a>0 \ \mbox{in} \ (0,1], \ a(0)=0$;
    \item $\exists K\in [1,2) \mbox{ such that } \ xa'(x)\leq Ka(x) \ \forall x\in[0,1]$;
    \item $\left\{\begin{array}{cc}\exists\theta\in(1,K]; \ x\longmapsto\frac{a(x)}{x^\theta} \ \mbox{is nondecreasing near} \ 0,& \mbox{if} \ K>1;\\ \exists\theta\in(0,1) \ x\longmapsto\frac{a(x)}{x^\theta} \ \mbox{is nondecreasing near} \ 0,& \mbox{if} \ K=1.\end{array}\right.$
\end{enumerate} 
The description above has the function $a(x)=x^{\alpha}$ as a prototype, where $\alpha \in (0,1)$ for the (WDC) and $\alpha \in [1,2)$ for the (SDC). 
 In this particular scenario, some Carleman estimates are presented in \cite{cannarsa2006null}, which imply null controllability results by using the {\it Hilbert's Uniqueness Method} (HUM). Additionally, the mentioned work also establishes that the {\it super strongly degenerate problem} ($\alpha \geq 2$) is not null controllable, in general.

Later, in \cite{fragnelli2013carleman}, similar results were achieved even when the degeneracy occurs within an interior point of $(0,1)$. In that work, it is considered that there exists $x_0\in(0,1)$ such that the degenerate function $a\in\mathcal{C}^1([0,1]-\{x_0\})$ satisfies $a(x_0)=0$ and $a>0$ in $[0,1]-\{x_0\}$. Besides that, the function $a$ must also satisfy one of the two conditions: 

\begin{itemize}
\item[(a)]  $\exists K\in(0,1)$ such that $(x-x_0)a'(x)\leq Ka(x) \ \forall x\in[0,1]-\{x_0\}$;
\item[(b)]  $a\in W^{1,\infty}(0,1)$ and $\exists K\in[1,2)$ such that $(x-x_0)a'(x)\leq Ka(x) \ \forall x\in[0,1]-\{x_0\}$,
\end{itemize}
where (a) represents a reformulation of the weakly degenerate case (RWDC), as well as, (b) is a reformulation of the strongly degenerate one (RSDC). 
A typical example of such a function is $a(x)=|x-x_0|^\alpha$, for $\alpha \in (0,2)$. It is worth noting that the null controllability theorems, as proven in \cite{fragnelli2013carleman}, rely on a geometric assumption over the control domain $\omega$, namely \begin{equation}
    \label{hipgeo1}x_0\in\omega.
\end{equation}

More recently, in \cite{faria2020carleman}, the third author of this current paper extended the investigation developed in \cite{fragnelli2013carleman}, dealing with second-order operators that can degenerate in a set of positive measure. To be more precise, it is taken into consideration $(a(x)u_x )_x$, where $a\equiv 0$ in an interval
$[A,B]\subset(0,1)$ and the geometric control domain condition \begin{equation}
    \label{hipgeo2}[A,B]\subset\omega
\end{equation}
is imposed. It is important to point out here that \eqref{hipgeo2} is a natural adaptation of \eqref{hipgeo1} in this context. Furthermore, this kind of assumption is not considered in  \cite{cannarsa2006null}, where the authors established that the null controllability property does not hold for $\alpha\geq2$.

The main novelty of this paper is to extend the investigation of \cite{faria2020carleman} and \cite{fragnelli2013carleman} for the super strongly degenerate case. Here we still assume the geometrical assumption \eqref{hipgeo2} and also consider the following additional regularity hypotheses related to $a=a(x)$:

%Even though the problem studied in \cite{cannarsa2006null} is not null controllable for $\alpha \geq 2$, in this work, we will extend the investigation of \cite{faria2020carleman} and \cite{fragnelli2013carleman} for the super strongly degenerate case, under the additional assumption \eqref{hipgeo2}.
%that the problems studied by [FRAGNELLI 2013] and \cite{faria2020carleman} are null controllable for the super strongly degenerate case under the same geometric assumptions that they made.
%Also, we consider the following additional regularity hypotheses related to $a=a(x)$:

\begin{equation}\label{hipreg}
\frac{1}{a} \notin L^1 ([0,A)\cup(B,1]), \\
    a\in W^{2,\infty}(0,1) \ \ \mbox{and} \ \ 
    a a_{xx}\in W^{1,\infty}(0,1).
\end{equation}
\

A similar condition to the first one presented in \eqref{hipreg} is also found in \cite{fragnelli2012, fragnelli2015}, and it plays a crucial role in demonstrating the existence of a solution to the problem at hand. This condition is naturally true under the strongly degenerate case (SDC) assumptions. However, it is worth noting that, for the super strong case, this condition excludes the possibility of the compact embedding of the space $H^1_a(0,1)$, which is defined ahead, into the space $L^2(0,1)$. This observation has been made in Cannarsa and Fragnelli's work \cite{cannarsa2006null}, specifically for the case where $a(x)=x^{\alpha}, \alpha\geq2$, and it can be easily extended to the case where $a(x)=(x-x_0)^{\alpha}, \alpha\geq2$. This limitation becomes an obstacle when studying the controllability of the semilinear parabolic problem associated with \eqref{prob0} since, in general, fixed point methods are employed to obtain the controllability of semilinear problems.

In the sequel, the whole discussion assumes that $a=a(x)$ degenerates in $[A,B]\subset \omega \subset (0,1)$, that is, 

\begin{equation}\label{hipdeg}
\begin{cases} 
 1. \  a(x)=0 \ \mbox{in} \  [A,B];\\
2. \ a(x)>0 \ \mbox{in} \  [0,A)\cup(B,1],
\end{cases} 
\end{equation}
and \eqref{hipreg} holds. For example, the function $a:[0,1]\longrightarrow \R$, given by

\begin{equation*}
a(x)=\begin{cases}
(A-x)^\alpha, & x\in [0,A),\\
0, & x\in [A,B],\\
(x-B)^\beta, & x\in (B,1],
\end{cases}
\end{equation*}
with $\alpha,\beta\geq 2$, fulfills the properties that describe the super strongly degenerate condition. Also, we are supposed to emphasize that our approach extends \cite{faria2020carleman} and \cite{fragnelli2013carleman} to the super-strongly degenerate range, including the situation $A=B$.

%Hence, in what follows we assume that the function $a$ has the regularities properties in \eqref{hipreg} and degenerates in a interval $[A,B]\subset(0,1)$. That means 

%Here we assume that $a$ degenerates in a interval  $[A,B]\subset \omega$, i.e., \st{there exists $d_0\geq0$ and $x_0\in(0,1)$ such that $D_0=[x_0-d_0,x_0+d_0]\subset\subset(0,1)$ and }

%\begin{equation}\label{hipdeg}
%\begin{cases} 
% 1. \  a(x)=0 \ \mbox{in} \  [A,B];\\
%2. \ a(x)>0 \ \mbox{in} \  [0,A)\cup(B,1].
%\end{cases} 
%\end{equation}

%Note that we do not assumed that $A<B$. That way, considering the case $A=B$, we will extend both the results presented in [FRAGNELLI 2013] and in \cite{faria2020carleman} to the super strong degenerate case.

At this point, we are ready to present the main result of this paper:

\begin{thm}\label{main}    Assume that the function $a$ satisfies \eqref{hipreg} and \eqref{hipdeg}. If $\omega$ satisfies \eqref{hipgeo2}, then the system \eqref{prob0} is null controllable.\end{thm}

The remainder of this work is dedicated to obtain Carleman and observability estimates for the adjoint equations associated with the equation \eqref{prob0}, following the ideas presented in \cite{araujo2023carleman}. This is accomplished in Section \ref{carleman}. As a result, we apply these estimates to prove Theorem \ref{main} in Section \ref{NC}.

\section{Carleman and observability inequalities}\label{carleman}

This section is devoted to the obtainment of a Carleman-type estimate, which will lead us to our null controllability results. Such an inequality is valid for any solution of
\begin{equation}\label{prob0'}
\begin{cases}
v_t+(a(x) v_x)_x -c(x,t)u =  h &\text{in }    Q,\\ v(0,t)=v(1,t)=0   & \text{in }  (0,T),\\ v(x,T)= v_T(x)  & \text{in } (0,1),
\end{cases}
\end{equation}
which is the adjoint system associated with the linearization of \eqref{prob0}.

Before proving our Carleman and observability inequalities, we will present some spaces and well-posedness results for \eqref{prob0}.

Let us consider the spaces 
\begin{equation*}
\begin{split}
H_a^1:= \{  u\in L^2(0,1);\ u\text{ is {locally} absolutely continuous in } [0,1],\ \\
 (a^{1/2} u_x)(x)\in L^2(0,1) \mbox{ and } u(0)=u(1)=0\}
\end{split}
\end{equation*}
and 
\[
H_{\alpha}^2:= \{  u\in H_{\alpha}^1;\ a^{1 /2} u_x\in H^1(0,1) \}
\]
endowed, respectively, with the norms
\[ \|u\|_{H_a^1}:=\left( \|u\|_{L^2(0,1)}^2+\| a^{1/2} u_x\|_{L^2(0,1)}^2 \right) ^{1/2} ;\]
and
$\|u\|_{H_{\alpha}^2}:=\left( \|u\|_{H_{\alpha}^1}^2+\|(a u_x)_x\|_{L^2(0,1)}^2 \right) ^{1/2}$.

%Another important space in this context is $H_\alpha^{-1}=(H_\alpha^1)'$ (the dual space of $H_\alpha^1$). For  $u\in H_\alpha^{-1}$, from Lax-Milgram Theorem, there exists a unique  $\tilde{u}\in H_\alpha^{1}$ such that \[\langle u,v\rangle_{H_\alpha^{-1}}=\int_0^1x^\alpha \tilde{u}_xv_x\,dx \ \ \forall v\in H_\alpha^1.\]Hence, $H_\alpha^{-1}$ is a Hilbert space equipped with the inner product $$(u,v)_{H_\alpha^{-1}}=\int_0^1x^\alpha\tilde{u}_x v_x\,dx.$$

The following well-posedness result for \eqref{prob0} was obtained in \cite[Theorem 2.2]{faria2020carleman} for the weakly and strongly degenerate cases. However, following the same ideas presented there, we can prove:

%  \begin{equation}\label{prob1}
% \begin{cases}
% u_t-(a(x) u_x)_x =  f1_{\omega} &\text{in }    Q:=(0,1)\times(0,T),\\ u(0,t)=u(1,t)=0   & \text{in }  (0,T),\\ u(x,0)= u_0(x)  & \text{in } (0,1),
% \end{cases}
% \end{equation}
%  stated below:
\begin{prop}\label{well-pos}
Assume that $a=a(x)$ satisfies  \eqref{hipreg} and \eqref{hipdeg}.	Given  $f\in L^1(0,T;L^2(0,1))$ and $u_0\in L^2(0,1)$, there exists a unique weak solution $u\in C^0([0,T];H^1_a)\cap C^1([0,T];L^2(0,1))$ of \eqref{prob0}.	In addition, there exists a positive constant $C_{T,a}$ such that
	\begin{equation}\label{ineq1}
	\sup_{t\in[0,T]} \|u(t)\|^2_{L^2(0,1)}+\int_0^T\|u(t)\|^2_{H_a^1}\,dt 
	\leq C_{T,a}\left(\|f\|^2_{L^2(Q)} +\|u_0\|^2_{L^2(0,1)}\right).
	\end{equation}
\end{prop}

Now, under the assumptions described in \eqref{hipgeo2}, let us consider $x_0=(A+B)/2$ and $\delta>0$ such that 
\[
[A,B]\subset \omega_\delta \subset\subset \omega ,
\]
%For instance, we can take $x_0=(A+B)/2$ and $\delta$ greater than and sufficiently close to $(B-A)/2$.
where $\omega_\delta=(x_0-\delta,x_0+\delta)$.  Since $a\in W^{2,\infty}(0,1)\hookrightarrow C^1([0,1])$, let $m_\delta >0$ be the minimum of
\[
\displaystyle \{ a(x); x\in [0,1]-\omega_\delta \},
\]
that is, 
%the minimum value of $a=a(x)$ in $[0,1]-\omega_\delta$, that is,
\begin{equation}
\label{hip1}a(x)\geq m_\delta \ \ \text{ for any } \ \ x\in [0,1]-\omega_\delta.
\end{equation}

%The adjoint system associated to the linearization of \eqref{prob0} is given by  

As usual in the Carleman method, for each $\lambda >0$, we start introducing a set of weight functions, as below:
%Now, for $\lambda\geq\lambda_0>0$, we will introduce a set of weight functions
\begin{multline}\label{pesos}
\theta(t)=\frac{1}{[t(T-t)]^4}, \ \ \eta(x)=-\frac{(x-x_0)^2}{2}, \ \ \xi(x,t)=\theta(t)e^{\lambda(2|\eta|_\infty+\eta(x))}\\  \mbox{and} \ \ \sigma(x,t)=\theta(t)e^{4\lambda|\eta|_\infty}-\xi(x,t).
\end{multline}

From \eqref{hip1}, we observe that there exists $C>0$ such that \begin{equation}\label{hip3}
|\eta'(x)|a(x)\geq C \ \ \mbox{for any} \ \ x\in [0,1]-\omega_\delta .
\end{equation}

In light of the previous explanation, we can now obtain the desired Carleman estimate, as stated in the following theorem.

\begin{thm}[Carleman estimate]\label{th.carleman}
 Assume that the function $a$ satisfies \eqref{hipreg} and \eqref{hipdeg}. If $\omega$ satisfies \eqref{hipgeo2}, then there exist positive constants $C$, $s_0$ and $\lambda_0$, only depending on $T$, $a$, $c$ and $\omega$, such that, for any $s\geq s_0$, $\lambda \geq \lambda_0$ and $v$ solution of \eqref{prob0'}, we have
\begin{align}\label{eq18a}
\iint_Q e^{-2s\sigma} \left[s^{-1}\lambda^{-1}\xi^{-1}\left(|v_t|^2+ \left|(a(x)v_x)_x\right|^2\right)+s^3\lambda^4\xi^3|v|^2+s\lambda^2\xi a(x)|v_x|^2\right]\,dx\,dt\nonumber\\ \leq C\left[\|e^{-s\sigma}h\|_2^2+s^3\lambda^4\int_0^T\int_{\omega_\delta}e^{-2s\sigma}\xi^3|v|^2\,dx\,dt\right].
\end{align}

\end{thm}

Notice that it suffices to prove Theorem \ref{th.carleman} for $c=0$, since the general case follows taking $\tilde{h}=h+cv$.

Indeed, for each $s>0$, we will consider the change of variables 
\[
z=e^{-s\sigma}v.
\]
We notice that
$z=0$ in $\partial Q$, and simple computations give us 
\[v_t=e^{s\sigma}(s\sigma_tz+z_t)\]
and
\[
(a(x)v_x)_x=e^{s\sigma}[s^2|\sigma_x|^2a(x)z+2s\sigma_xa(x)z_x+s(\sigma_xa(x))_xz+(a(x)z_x)_x].
\]
Consequently, from \eqref{prob0'}, it is clear that $z$ is a solution of
\begin{equation}\label{eq1}
\begin{cases}
P^+z+P^-z=G, & \text{ in } Q,\\
z=0, & \text{ in } \partial Q,
\end{cases}
\end{equation}
where 
\[
\begin{array}{cc}P^-z:=2s(\sigma_xa(x))_xz+2s\sigma_xa(x)z_x+z_t:=I_{11}+I_{12}+I_{13},
\\ P^+z:=s^2|\sigma_x|^2a(x)z+[a(x)z_x]_x+s\sigma_tz:=I_{21}+I_{22}+I_{23}\end{array}
\]
and 
\[
G=e^{-s\sigma}h+s(\sigma_xa(x))_xz.
\]
By using \eqref{eq1}, we can arrive at \begin{equation}
\label{eq2}\|P^-z\|_2^2+\|P^+z\|_2^2+2(\!(P^-z,P^+z)\!)=\|G\|_2^2,\end{equation}
where the norm in  $L^2(Q)$ will be denoted by  $\|\cdot\|_2$ and its inner product by $(\!(\cdot,\cdot)\!)$. 

Next, we will deal with each term on the left side of \eqref{eq2}, in order to achieve the following result.
\begin{prop}
 Assume that the function $a$ satisfies \eqref{hipreg} and \eqref{hipdeg}. If $\omega$ satisfies \eqref{hipgeo2}, then there exist positive constants $C$, $s_0$ and $\lambda_0$, only depending on $T$, $a$, $c$ and $\omega$, such that, for any $s\geq s_0$, $\lambda \geq \lambda_0$ and $z$ solution of \eqref{eq1}, we have
\begin{align}\label{eq18}
\iint_Q[s^{-1}\xi^{-1}[|z_t|^2+|[a(x)z_x]_x|^2]+s^3\lambda^4\xi^3|z|^2+s\lambda^2\xi a(x)|z_x|^2]\,dx\,dt\nonumber\\ \leq C\left[\|e^{-s\sigma}h\|_2^2+\int_0^T\int_{\omega_\delta}s^3\lambda^4\xi^3|z|^2\,dx\,dt\right].
\end{align}
\end{prop}

\begin{proof}
The beginning of this proof is focused on estimating $(\!(P^-z,P^+z)\!)$. 
\medskip

\textbf{\underline{Part I}:} Estimate for $(\!(P^-z, I_{21})\!)$.

Since
\[\xi_x=\lambda \eta'\xi \text{ and } \sigma_x=-\xi_x=-\lambda\eta'\xi,\]
we take
\[(\!(I_{11},I_{21})\!)=\iint_Q[-2s^3\lambda^4\xi^3|\eta'|^4|a(x)|^{{\color{black} 2}}|z|^2-2s^3\lambda^3\xi^3|\eta'|^2{\color{black}a(x)}(\eta'a(x))_x|z|^2]\,dx\,dt,\]
\begin{align*}
(\!(I_{12},I_{21})\!)&=s^3\iint_Q|\sigma_x|^3|a(x)|^2(|z|^2)_x\,dx\,dt\\ &=3s^3\lambda^4\iint_Q\xi^3|\eta'|^4|a(x)|^2|z|^2\,dx\,dt+s^3\lambda^3\iint_Q\xi^3({{\color{black}(\eta')^3a^2(x)}})_x|z|^2\,dx\,dt
\end{align*}
and
\[(\!(I_{13},I_{21})\!)=-s^2\lambda^2\iint_Q\xi\xi_t|\eta'|^2a(x)|z|^2\,dx\,dt.\]
Thus 
\begin{multline*}
(\!(P^-z,I_{21})\!)= s^3\lambda^4\iint_Q\xi^3|\eta'|^4a^{\color{black}2}(x)|z|^2\,dx\,dt\\+s^3\lambda^3 \iint_Q\xi^3[({\color{black}(\eta')^3a^2}(x))_x-2|\eta'|^2{\color{black}a(x)}(\eta' a(x))_x]|z|^2\,dx\,dt\\ -s^2\lambda^2\iint_Q\xi\xi_t|\eta'|^2a(x)|z|^2\,dx\,dt.
\end{multline*} 

Recalling that $a\in W^{2,\infty} (0,1)\hookrightarrow C^1([0,1])$, we know that 
\[({\color{black}(\eta')^3a^2}(x))_x- 2|\eta'|^2{\color{black}a(x)}(\eta' a(x))_x \text{ and } |\eta'|^2 a(x) \text{ are bounded}.\]
At the same time,  we have $|\xi\xi_t|\leq C\xi^3$. Therefore, for sufficiently large   $\lambda$ and $s$, we can deduce that 
\[(\!(P^-z,I_{21})\!)\geq s^3\lambda^4\iint_Q\xi^3|\eta'|^4a^2(x)|z|^2\,dx\,dt-Cs^3\lambda^3\iint_Q\xi^3|z|^2\,dx\,dt.\]

Furthermore, from \eqref{hip3}, we obtain  
\begin{align*}
s^3\lambda^4\iint_Q\xi^3|\eta'|^4a(x)|z|^2\,dx\,dt&\geq  Cs^3\lambda^4\int_0^T\int_{[0,1]-\omega_\delta}\xi^3|\eta'|^4a(x)|z|^2\,dx\,dt\\ &\geq Cs^3\lambda^4\int_0^T\int_{[0,1]-\omega_\delta}\xi^3|z|^2\,dx\,dt\\ &= Cs^3\lambda^4\iint_Q\xi^3|z|^2\,dx\,dt-Cs^3\lambda^4\int_0^T\int_{\omega_\delta}\xi^3|z|^2\,dx\,dt.
\end{align*}

Thus, if $s$ and $\lambda$ are sufficiently large,
\begin{equation}
\label{eq3}(\!(P^-z, I_{21})\!)\geq Cs^3\lambda^4\iint_Q\xi^3|z|^2\,dx\,dt- Cs^3\lambda^4\int_0^T\int_{\omega_\delta}\xi^3|z|^2\,dx\,dt
\end{equation}
holds. 
\medskip 

\textbf{\underline{Part II}:} Estimate for $(\!(P^-z, I_{23})\!)$. 

Let us observe that
\[(\!(I_{11},I_{23})\!)=-2s^2\lambda\iint_Q\xi\sigma_t[\lambda|\eta'|^2a(x)+(\eta'a(x))_x]|z|^2\,dx\,dt,\]
\[(\!(I_{12},I_{23})\!)=s^2\lambda\iint_Q\xi[\lambda(\sigma_t-\xi_t)|\eta'|^2a(x)+\sigma_t(\eta'a(x))_x]|z|^2\,dx\,dt\]
and
\[(\!(I_{13},I_{23})\!)=-\frac{s}{2}\iint_Q\sigma_{tt}|z|^2\,dx\,dt.\]
Since $a\in C^1([0,1])$,  $|\xi_t|,|\sigma_t|\leq C\xi^2$ and $\sigma_{tt}\leq C\xi^3$, we can proceed as before and  use \eqref{eq3} to deduce that 
\begin{equation}
\label{eq4}(\!(P^-z, I_{21}+I_{23})\!)\geq Cs^3\lambda^4\iint_Q\xi^3|z|^2\,dx\,dt- Cs^3\lambda^4\int_0^T\int_{\omega_\delta}\xi^3|z|^2\,dx\,dt.
\end{equation}

\textbf{\underline{Part III}:} Estimate for $(\!(P^-z, I_{22})\!)$. 

Initially, we have 
\begin{equation}
\label{eq5}(\!(I_{11},I_{22})\!)=-2s\lambda\iint_Q[\lambda\xi|\eta'|^2a(x)z[a(x)z_x]_x+\xi[\eta'a(x)]_xz[a(x)z_x]_x]\,dx\,dt.
\end{equation}

For the first term on the right side of \eqref{eq5}, we can write
\begin{align*}-2s\lambda^2\iint_Q\xi|\eta'|^2a(x)z[a(x)z_x]_x\,dx\,dt&=2s\lambda^3\iint_Q\xi(\eta')^3|a(x)|^2zz_x\,dx\,dt\\ &+2s\lambda^2\iint_Q\xi[|\eta'|^2a(x)]_xza(x)z_x\,dx\,dt\\ &+2s\lambda^2\iint_Q\xi|\eta'|^2|a(x)|^2|z_x|^2\,dx\,dt\\
&=:J_1 +J_2++2s\lambda^2\iint_Q\xi|\eta'|^2|a(x)|^2|z_x|^2\,dx\,dt,
\end{align*}
where
\begin{align*}
J_1=2s\lambda^3\iint_Q\xi(\eta')^3|a(x)|^2zz_x\,dx\,dt&= -s\lambda^3\iint_Q\xi[\lambda|\eta'|^4|a(x)|^2+(|\eta'|^3|a(x)|^2)_x]|z|^2\,dx\,dt
\end{align*}
  and 
\begin{multline*}
J_2 =2s\lambda^2\iint_Q\xi\big(|\eta'|^2a(x)\big)_xazz_x\,dx\,dt\\=-s\lambda^2\iint_Q\xi\Bigg(\lambda\eta'a(x)\left[|\eta'|^2a(x)\right]_x+\Big(\left[|\eta'|^2a(x)\right]_xa(x)\Big)_x\Bigg)|z|^2\,dx\,dt.
\end{multline*}

Likewise, for the second term on the right side of \eqref{eq5}, we have
\begin{align*}
-2s\lambda\iint_Q\xi[\eta'a(x)]_xz[a(x)z_x]_x\,dx\,dt&=2s\lambda^2\iint_Q\xi a(x)\eta'[\eta'a(x)]_xzz_x\,dx\,dt\\ &+2s\lambda\iint_Q {\color{black}\xi a(x)} [\eta'a(x)]_{xx}zz_x\,dx\,dt\\ &+2s\lambda\iint_Q\xi[\eta'a(x)]_xa(x)|z_x|^2\,dx\,dt \\
&=: J_3+J_4+2s\lambda\iint_Q\xi[\eta'a(x)]_xa(x)|z_x|^2\,dx\,dt,
\end{align*}
where
\begin{align*}
J_3&=2s\lambda^2\iint_Q\xi a(x)\eta'[\eta'a(x)]_xzz_x\,dx\,dt\\
&=-s\lambda^2\iint_Q\xi[\lambda|\eta'|^2a(x)[\eta'a(x)]_x+[\eta'a(x)[\eta'a(x)]_x]_x]|z|^2\,dx\,dt
\end{align*}
and
\begin{align*}
J_4&=2s\lambda\iint_Q {\color{black}\xi a(x)} [\eta'a(x)]_{xx}zz_x\,dx\,dt\\ 
&=-s\lambda\iint_Q\xi[\lambda\eta'a(x)[\eta'a(x)]_{xx}+[[\eta'a(x)]_{xx}a(x)]_x]|z|^2\,dx\,dt.
\end{align*}

Combining these estimates we can conclude that
\begin{align}\label{eq6}
(\!(I_{11},I_{22})\!)&\geq-Cs\lambda^4\iint_Q\xi|z|^2\,dx\,dt+2s\lambda^2\iint_Q\xi|\eta'|^2|a(x)|^2|z_x|^2\,dx\,dt\nonumber\\ &+2s\lambda\iint_Q\xi[\eta'a(x)]_xa(x)|z_x|^2\,dx\,dt.
\end{align}

Arguing as before we deduce that
\begin{align}\label{eq7}
& 2s\lambda^2\iint_Q\xi|\eta'|^2|a(x)|^2|z_x|^2\,dx\,dt\\ 
&\quad \geq s\lambda^2\iint_Q\xi a(x)|z_x|^2\,dx\,dt-Cs\lambda^2\int_0^T\int_{\omega_\delta}\xi a(x)|z_x|^2\,dx\,dt.
\end{align}

Furthermore,
\begin{align}\label{eq8}
(\!(I_{13},I_{22})\!)&=-\iint_Qz[a(x)z_{xt}]_x\,dx\,dt=\iint_Qz_xa(x)z_{xt}\,dx\,dt\nonumber\\ &=\frac{1}{2}\iint_Q[a(x)|z_x|^2]_t\,dx\,dt=0.
\end{align}
Hence, from \eqref{eq6}-\eqref{eq8} we obtain 
\begin{align}\label{eq9}
(\!(I_{11}+I_{13},I_{22})\!)&\geq-Cs\lambda^4\iint_Q\xi|z|^2\,dx\,dt-Cs\lambda^2\int_0^T\int_{\omega_\delta}\xi a(x)|z_x|^2\,dx\,dt\nonumber\\ &+Cs\lambda^2\iint_Q\xi a(x)|z_x|^2\,dx\,dt.
\end{align}

Finally,
\begin{align}\label{eq10}
(\!(I_{12},I_{22})\!)&=s\lambda^2\iint_Q\xi|\eta'|^2|a(x)|^2|z_x|^2\,dx\,dt+s\lambda\iint_Q\xi\eta''|a(x)|^2|z_x|^2\,dx\,dt\nonumber\\&-s\lambda\int_0^T\xi\eta'|a(x)|^2|z_x|^2|_{x=0}^{x=1}\,dt
\end{align}

Since $\eta'(1)<0$ and $\eta'(0)>0$, the boundary term on \eqref{eq10} is $\geq0$. The other terms can be controlled as before. This led us to 
\begin{align}\label{eq11}
(\!(P^-,I_{22})\!)&\geq-Cs\lambda^4\iint_Q\xi|z|^2\,dx\,dt-Cs\lambda^2\int_0^T\int_{\omega_\delta}\xi a(x)|z_x|^2\,dx\,dt\nonumber\\ &+Cs\lambda^2\iint_Q\xi a(x)|z_x|^2\,dx\,dt.
\end{align}

Combining \eqref{eq4} and \eqref{eq11} we conclude that
\begin{align}\label{eq12}
(\!(P^-z,P^+z)\!)&\geq C\iint_Q[s^3\lambda^4\xi^3|z|^2+s\lambda^2\xi a(x)|z_x|^2]\,dx\,dt\nonumber\\& -C\int_0^T\int_{\omega_\delta}[s^3\lambda^4\xi^3|z|^2+s\lambda^2\xi a(x)|z_x|^2]\,dx\,dt.
\end{align}

From \eqref{eq2} and \eqref{eq12} we have that
\begin{align}\label{eq13}
\|P^-z\|_2^2+\|P^+z\|_2^2+\iint_Q[s^3\lambda^4\xi^3|z|^2+s\lambda^2\xi a(x)|z_x|^2]\,dx\,dt\nonumber\\ \leq C\left[\|G\|_2^2+\int_0^T\int_{\omega_\delta}[s^3\lambda^4\xi^3|z|^2+s\lambda^2\xi a(x)|z_x|^2]\,dx\,dt\right].
\end{align}

Now, using the definitions of $P^-z$, $P^+z$ and $G$ we obtain that
\begin{align}\label{eq14}
s^{-1}\iint_Q\xi^{-1}|z_t|^2\,dx\,dt&\leq s^{-1}\|P^-z\|_2^2+Cs\lambda^4\iint_Q\xi|z|^2\,dx\,dt+Cs\lambda\iint_Q\xi a(x)|z_x|^2\,dx\,dt\nonumber\\ &\leq C\left[\|G\|_2^2+\int_0^T\int_{\omega_\delta}[s^3\lambda^4\xi^3|z|^2+s\lambda^2\xi a(x)|z_x|^2]\,dx\,dt\right], 
\end{align}

\begin{align}\label{eq15}
& s^{-1}\iint_Q\xi^{-1}|[a(x)z_x]_x|^2\,dx\,dt\\ 
&\quad \leq s^{-1}\|P^+z\|_2^2+Cs^3\lambda^4\iint_Q\xi^3|z|^2\,dx\,dt+Cs\iint_Q\xi |z|^2\,dx\,dt\nonumber\\ 
&\quad \leq C\left[\|G\|_2^2+\int_0^T\int_{\omega_\delta}[s^3\lambda^4\xi^3|z|^2+s\lambda^2\xi a(x)|z_x|^2]\,dx\,dt\right],
\end{align}
and
\begin{align}\label{eq16}
\|G\|_2^2\leq \|e^{-s\sigma}h\|_2^2+Cs^2\lambda^4\iint_Q\xi^2|z|^2\,dx\,dt.
\end{align}

By combining \eqref{eq13}-\eqref{eq16} we can deduce that
\begin{align}\label{eq17}
\iint_Q[s^{-1}\xi^{-1}[|z_t|^2+|[a(x)z_x]_x|^2]+s^3\lambda^4\xi^3|z|^2+s\lambda^2\xi a(x)|z_x|^2]\,dx\,dt\nonumber\\ \leq C\left[\|e^{-s\sigma}h\|_2^2+\int_0^T\int_{\omega_\delta}[s^3\lambda^4\xi^3|z|^2+s\lambda^2\xi a(x)|z_x|^2]\,dx\,dt\right].
\end{align}

Since $\omega_\delta\subset\subset\omega$, we can take $\omega_\delta\subset\subset D_2\subset\subset\omega$ and a cut-off function $\varphi\in C^\infty([0,1])$ such that $0\leq\varphi\leq1$, $\varphi=1$ in $\omega_\delta$ and $\varphi=0$ in $[0,1]-D_2$. Hence, for $\varepsilon>0$, we have that
\begin{align*}
& s\lambda^2\int_0^T\int_{\omega_\delta}\xi a(x)|z_x|^2\,dx\,dt\\
&\leq s\lambda^2\int_0^T\int_{D_2}\xi\varphi a(x)z_xz_x\,dx\,dt\\
&=-s\lambda^2\int_0^T\int_{D_2}
 \left(\lambda\xi\eta'\varphi a(x)zz_x+\xi\varphi'a(x)zz_x+\xi\varphi[a(x)z_x]_xz\right)\,dx\,dt,
\end{align*}
\begin{align*}
-s\lambda^3\int_0^T\int_{D_2}\xi\eta'\varphi a(x)zz_x dx\,dt \leq C\iint_Q[s^2\lambda^4\xi^2|z|^2+\lambda^2a(x)|z_x|^2]\,dx\,dt,
\end{align*}
\begin{align*}
-s\lambda^2\int_0^T\int_{D_2} \xi\varphi'a(x)zz_x dx\,dt\leq C\iint_Q[s^2\lambda^4\xi^2|z|^2+a(x)|z_x|^2]\,dx\,dt
\end{align*}
and
\begin{align*}
& -s\lambda^2\int_0^T\int_{D_2}\xi\varphi[a(x)z_x]_xz\,dx\,dt\\
& \quad \leq C\varepsilon^{-1}s^3\lambda^4\int_0^T\int_\omega\xi^3|z|^2\,dx\,dt+\varepsilon s^{-1}\iint_Q\xi^{-1}|[a(x)z_x]_x|^2\,dx\,dt.
\end{align*}

Taking $\varepsilon >0$ sufficiently small, these last estimates together with \eqref{eq17} give us \eqref{eq18}.
%\begin{align}\label{eq18}
%\iint_Q[s^{-1}\xi^{-1}[|z_t|^2+|[a(x)z_x]_x|^2]+s^3\lambda^4\xi^3|z|^2+s\lambda^2\xi a(x)|z_x|^2]\,dx\,dt\nonumber\\ \leq C\left[\|e^{-s\sigma}h\|_2^2+\int_0^T\int_{\omega_\delta}s^3\lambda^4\xi^3|z|^2\,dx\,dt\right].
%\end{align}
\end{proof}

Coming back to the original variable $v$ we obtain Theorem \ref{th.carleman}.

    \begin{cor}[observability inequality]\label{cor.obs}
Assume that the function $a$ satisfies \eqref{hipreg} and \eqref{hipdeg}. If $\omega$ satisfies \eqref{hipgeo2}, then there exists a constant $C>0$ such that, for any $v_T\in L^2(0,1)$ and $v$ solution of \eqref{prob0'} with $h=0$, one has
	\begin{equation}
    	\label{obs}\|v(\cdot,0)\|_{L^2(0,1)}^2\leq C\iint_{\omega_T}e^{-2s\sigma}\xi^3|v|^2\,dx\,dt,
	\end{equation}
 {where we recall that $\omega_T=\omega\times (0,T)$.}
\end{cor}

\begin{proof}
	From Theorem \ref{th.carleman}, since $\omega_\delta\subset \omega$, we have that 
\begin{equation}\label{eq5.0}
    s^3\lambda^4\iint_Qe^{-2s\sigma}\xi^3|v|^2\,dx\,dt\leq Cs^3\lambda^4\int_0^T\int_\omega e^{-2s\sigma}\xi^3|v|^2\,dx\,dt.
\end{equation}	
	
Multiplying the equation in \eqref{prob0'} by $v$ and integrating on $(0,1)$ we obtain that
\[-\frac{1}{2}\frac{d}{dt}\|v(\cdot,t)\|_{L^2(0,1)}^2+\int_0^1 a|v_x|^2\,dx=-\int_0^1c|v|^2\,dx.\]
Hence,
    \[-\frac{1}{2}\frac{d}{dt}\|v(\cdot,t)\|_{L^2(0,1)}^2+\frac{1}{2}\int_0^1 a|v_x|^2\,dx\leq C\|v(\cdot,t)\|_{L^2(0,1)}^2.\] 
Thus,
\begin{equation}\label{eq5.1}
    \|v(\cdot,0)\|_{L^2(0,1)}^2\leq e^{2Ct}\|v(\cdot,t)\|_{L^2(0,1)}^2 \ \ \ \forall t\in(0,T).
\end{equation}
	
Integrating \eqref{eq5.1} on $(T/4,3T/4)$ and using \eqref{eq5.0} we deduce that
\begin{align*}
    \|v(\cdot,0)\|_{L^2(0,1)}^2&=\frac{2}{T}\int_{T/4}^{3T/4}\|v(\cdot,0)\|_{L^2(0,1)}^2\,dt\leq C\int_{T/4}^{3T/4}\int_0^1|v|^2\,dx\,dt\\ 
    &\leq  C\int_{T/4}^{3T/4}\int_0^1s^3\lambda^4e^{-2s\sigma}\xi^3|v|^2\,dx\,dt\leq C\int_0^T\int_\omega e^{-2s\sigma}\xi^3|v|^2\,dx\,dt.
\end{align*}
\end{proof}

\section{Null controllability for the linear problem}\label{NC}

This section is dedicated to proving the null controllability of the linear problem \eqref{prob0}, stated in Theorem \ref{main}. The first step is to derive an approximate null controllability result.
%\begin{equation}\label{probl%}
%\begin{cases}
%u_t-(a(x) u_x)_x +c(x,t)u=  f1_{\omega} &\text{in }    Q,\\ u(0,t)=u(1,t)=0   & \text{in }  (0,T),\\ u(x,0)= u_0(x)  & \text{in } (0,1),
%\end{cases}
%\end{equation}
%where $c\in L^\infty(Q)$.

To establish this, let us set $\varepsilon>0$. For any $f\in L^2(Q)$, we define a functional
\[J_\varepsilon(f)=\frac{1}{2}\iint_Q|f|^2\,dx\,dt+\frac{1}{2\varepsilon}\int_0^1|u^f(x,T)|^2\,dx,\]
where $u^f$ is the weak solution of \eqref{prob0}. It is not difficult to see that $J_\varepsilon$ is continuous, strictly convex, and satisfies 
\[J_\varepsilon(f)\to\infty \ \ \ \mbox{as} \ \ \ \|f\|_{L^2(Q)}\to\infty.\]
Hence $J_\varepsilon$ has a unique critic point that is a global minimum. Let us denote this minimum by $f_\varepsilon$ and $u_\varepsilon:=u^{f_\varepsilon}$.

\begin{lem}
    If $\varphi_\varepsilon$ is the weak solution of the problem \begin{equation}\label{problc}
\begin{cases}
\varphi_{\varepsilon t}+(a(x) \varphi_{\varepsilon x})_x-c(x,t)\varphi_\varepsilon= 0 &\text{in }    Q,\\ \varphi_{\varepsilon} (0,t)=\varphi_{\varepsilon}(1,t)=0   & \text{in }  (0,T),\\ \varphi_\varepsilon(x,T)= \frac{1}{\varepsilon}u_\varepsilon(x,T)  & \text{in } (0,1),
\end{cases}
\end{equation}then $f_\varepsilon=-1_\omega\varphi_\varepsilon$.
\end{lem}

\begin{proof}
    For the sake of simplicity, let us set
\[Au:=(au_x)_x-cu \ \ \ \mbox{and} \ \ \ L_t(f)=\int_0^te^{(t-s)A}1_\omega f(x,s)\,ds.\] Then, $u^f$ is the solution of \eqref{prob0} if, and only if, \[u^f(x,t)=e^{tA}u_0(x)+L_t(f)(x).\]   

For $h\in L^2(Q)$, since \eqref{prob0} is linear, we have $u^{f+h}=u^f+z^h$, where $z^h$ is the solution of 
\begin{equation*}
\begin{cases}
z_t^h=Az^h+1_\omega h &\text{in }    Q,\\ z^h(0,t)=z^h(1,t)=0   & \text{in }  (0,T),\\ z^h(x,0)=0  & \text{in } (0,1).
\end{cases}
\end{equation*}
This leads us to \[J_\varepsilon(f+h)-J_\varepsilon(f)=\int_{0}^{T}\!\!\!\!\int_{0}^{1} [hf+\frac{1}{2}|h|^2]\,dx\,dt+\frac{1}{2\varepsilon}\int_0^12[u^f(x,T)z^h(x,T)+|z^h(x,T)|^2]\,dx.\]Thus
\begin{align*}
    J'_\varepsilon(f)h&=\int_0^T\int_0^1 fh\,dx\,dt+\frac{1}{\varepsilon}\int_0^1u^f(x,T)z^h(x,T)\,dx\\ &= \int_0^T\int_0^1 fh\,dx\,dt+\frac{1}{\varepsilon}\int_0^1u^f(x,T)\int_0^Te^{(T-t)A}1_\omega h\,dt\,dx\\ 
    &=\int_0^T\int_0^1 fh\,dx\,dt+\int_0^T\left\langle\frac{1}{\varepsilon}1_\omega u^f(x,T), e^{(T-t)A}h\right\rangle_{L^2(0,1)}dt\\&=\int_0^T\left[\langle  f,h\rangle_{L^2(0,1)}+\left\langle 1_\omega e^{(T-t)A^*}\left(\frac{1}{\varepsilon}u^f(x,T)\right),h\right\rangle_{L^2(0,1)}\right]dt.
\end{align*}
The result now follows from the fact that $J'_\varepsilon(f_\varepsilon)=0$.
\end{proof}

\begin{prop}\label{prop.3.2}
    For any $\varepsilon>0$, there exist $f_\varepsilon\in L^2(Q)$ and a corresponding solution $u_\varepsilon$ of \eqref{prob0} such that \[\iint_Q|f_\varepsilon|^2\,dx\,dt\leq C\|u_0\|^2_{L^2(0,1)} \ \ \ \mbox{and} \ \ \ \int_0^1|u_\varepsilon(x,T)|^2\,dx\leq C\varepsilon\|u_0\|^2_{L^2(0,1)}.\]
\end{prop}
\begin{proof}
    Let us consider $f_\varepsilon$ as the function given in the proof of the previous lemma. Multiplying the equation in \eqref{problc} by $u_\varepsilon$ and integrating in $Q$, we get \begin{equation*}
        \frac{1}{\varepsilon}\int_0^1|u_\varepsilon(x,T)|^2\,dx-\int_0^1\varphi_\varepsilon(x,0)u_0(x)\,dx=-\int\!\!\!\!\int_Q|f_\varepsilon|^2\,dx\,dt.
    \end{equation*}
    Therefore, \begin{equation}\label{eq5.2.1}
        \frac{1}{\varepsilon}\int_0^1|u_\varepsilon(x,T)|^2\,dx+ \int_{0}^{T} \int_{\omega}|f_\varepsilon|^2\,dx\,dt=\int_0^1\varphi_\varepsilon(x,0)u_0(x)\,dx.
    \end{equation} 
    The observability inequality \eqref{obs} lead us to \begin{align*}
        \int\!\!\!\!\int_Q|f_\varepsilon|^2\,dx\,dt&\leq\int_0^1\varphi_\varepsilon(x,0)u_0(x)\leq \|u_0\|_{L^2(0,1)}\left[C\int_0^T\int_\omega|\varphi_\varepsilon|^2\,dx\,dt\right]^{1/2}\\ &\leq C\|u_0\|_{L^2(0,1)}\left[\int_0^T\int_\omega|f_\varepsilon|^2\,dx\,dt\right]^{1/2}\\
        &\leq C\|u_0\|_{L^2(0,1)}\left[\int\!\!\!\!\int_Q|f_\varepsilon|^2\,dx\,dt\right]^{1/2}.
    \end{align*}
    Hence \[\int\!\!\!\!\int_Q|f_\varepsilon|^2\,dx\,dt\leq C\|u_0\|^2_{L^2(0,1)}.\]
    On the other hand, in a similar way, from \eqref{eq5.2.1} we obtain \begin{align*}
        \int_0^1|u_\varepsilon(x,T)|^2\,dx&\leq \varepsilon\int_0^1\varphi_\varepsilon(x,0)u_0(x)\,dx\leq\varepsilon\|u_0\|_{L^2(0,1)}.\left[C\int\!\!\!\!\int_Q|f_\varepsilon|^2\,dx\,dt\right]^{1/2}\\&\leq C\varepsilon\|u_0\|^2_{L^2(0,1)}.
    \end{align*}This concludes the proof.
\end{proof}

Now, we are ready to prove Theorem \ref{main}. Indeed, from inequality \eqref{ineq1} and Proposition \ref{prop.3.2}, combined with some standard arguments, we have
\begin{equation*}
\begin{cases}
    f_\ep \rightharpoonup f \text{ in } L^2(Q);\\
    u_\ep \rightharpoonup u \text{ in } L^2(0,T;H^1_a); \\
    u_{\ep t} \rightharpoonup u_t \text{ in } L^2(0,T;H^{-1}_a); \\
    u_\ep (\cdot ,T) \rightharpoonup u(\cdot ,T) \text{ in } L^2 (0,1).
\end{cases}
\end{equation*}
Now, taking $\gamma\in H_a^1$, multiplying the equation in \eqref{prob0} by $\gamma$ and integrating on $Q$, we obtain \begin{align*}
    \int_0^1u_\varepsilon(\cdot,T)\gamma \,dx-\int_0^1u_0\gamma\,dx+\iint_Qau_{\varepsilon x}\gamma_x\,dx\ dt+\iint_Qcu_\varepsilon\gamma\,dx\,dt=\iint_Q1_\omega f_\varepsilon \gamma \,dx\ dt.
\end{align*}Thus \[\int_0^1u_\varepsilon(\cdot,T)\gamma \,dx=\int_0^1u_0\gamma\,dx-\iint_Qau_{\varepsilon x}\gamma_x\,dx\ dt-\iint_Qcu_\varepsilon\gamma\,dx\,dt+\iint_Q1_\omega f_\varepsilon \gamma \,dx\ dt=:L_\varepsilon\]Using that $f_\varepsilon\rightharpoonup f$ in $L^2(Q)$ and $u_\varepsilon\rightharpoonup u$ in $L^2(0,T;H_a^1)$ we deduce that \[\lim_{\varepsilon\to0}L_\varepsilon =\int_0^1u_0\gamma\,dx-\iint_Qau_{x}\gamma_x\,dx\ dt-\iint_Qcu\gamma\,dx\,dt+\iint_Q1_\omega f \gamma \,dx\ dt=\int_0^1u(\cdot,T)\gamma\,dx.\] 
%In this last passage we use the fact that $u$ is the weak solution of \eqref{prob0}. %Passing to the limit and using the definition of weak solution we deduce that 
Therefore,\[\lim_{\varepsilon\to0}\int_0^1u_\varepsilon(\cdot,T)\gamma\,dx=\int_0^1u(\cdot,T)\gamma\,dx \ \ \forall \gamma\in H_a^1.\]
On the other hand, Proposition \ref{prop.3.2} gives us \[\left|\int_0^1u_\varepsilon(x,T)\gamma(x)\,dx\right|\leq \|u_\varepsilon(\cdot,T)\|_{L^2(0,1)}\|\gamma\|_{L^2(0,1)}\leq\varepsilon C\to0.\]Hence, \[\int_0^1u(\cdot,T)\gamma\,dx=0 \ \ \forall \gamma\in H_a^1\] and this lead us to $u(\cdot,T)=0$. So that, the proof of Theorem \ref{main} is concluded.
\bibliography{references}

\begin{thebibliography}{6}
\expandafter\ifx\csname natexlab\endcsname\relax\def\natexlab#1{#1}\fi
\providecommand{\url}[1]{\texttt{#1}}
\providecommand{\href}[2]{#2}
\providecommand{\path}[1]{#1}
\providecommand{\DOIprefix}{doi:}
\providecommand{\ArXivprefix}{arXiv:}
\providecommand{\URLprefix}{URL: }
\providecommand{\Pubmedprefix}{pmid:}
\providecommand{\doi}[1]{\href{http://dx.doi.org/#1}{\path{#1}}}
\providecommand{\Pubmed}[1]{\href{pmid:#1}{\path{#1}}}
\providecommand{\bibinfo}[2]{#2}
\ifx\xfnm\relax \def\xfnm[#1]{\unskip,\space#1}\fi
%Type = Article
\bibitem[{Ara{\'u}jo et~al.(2023)Ara{\'u}jo, Demarque \&
  Viana}]{araujo2023carleman}
\bibinfo{author}{Ara{\'u}jo, B.~S.}, \bibinfo{author}{Demarque, R.}, \&
  \bibinfo{author}{Viana, L.} (\bibinfo{year}{2023}).
\newblock \bibinfo{title}{Carleman inequality for a class of super strong
  degenerate parabolic operators and applications}.
\newblock {\it \bibinfo{journal}{Electronic Journal of Qualitative Theory of
  Differential Equations}\/},  {\it \bibinfo{volume}{2023}\/},
  \bibinfo{pages}{1--25}.
%Type = Article
\bibitem[{Cannarsa \& Fragnelli(2006)}]{cannarsa2006null}
\bibinfo{author}{Cannarsa, P.}, \& \bibinfo{author}{Fragnelli, G.}
  (\bibinfo{year}{2006}).
\newblock \bibinfo{title}{Null controllability of semilinear degenerate
  parabolic equations in bounded domains.}
\newblock {\it \bibinfo{journal}{Electronic Journal of Differential
  Equations}\/},  (pp. \bibinfo{pages}{1--20}).
%Type = Article
\bibitem[{Faria(2020)}]{faria2020carleman}
\bibinfo{author}{Faria, J.} (\bibinfo{year}{2020}).
\newblock \bibinfo{title}{Carleman estimates and observability inequalities for
  a class of problems ruled by parabolic equations with interior degenaracy}.
\newblock {\it \bibinfo{journal}{Applied Mathematics \& Optimization}\/},  (pp.
  \bibinfo{pages}{1--24}).
%Type = Article
\bibitem[{Fragnelli et~al.(2012)Fragnelli, Goldstein, Goldstein \&
  Romanelli}]{fragnelli2012}
\bibinfo{author}{Fragnelli, G.}, \bibinfo{author}{Goldstein, G.~R.},
  \bibinfo{author}{Goldstein, J.~R.}, \& \bibinfo{author}{Romanelli, S.}
  (\bibinfo{year}{2012}).
\newblock \bibinfo{title}{Generators with interior degeneracy on spaces of
  $l^2$ type}.
\newblock {\it \bibinfo{journal}{Electronic Journal of Differential
  Equations}\/},  {\it \bibinfo{volume}{2012}\/}, \bibinfo{pages}{1--30}.
%Type = Article
\bibitem[{Fragnelli et~al.(2015)Fragnelli, Marinoschi, Mininni \&
  Romanelli}]{fragnelli2015}
\bibinfo{author}{Fragnelli, G.}, \bibinfo{author}{Marinoschi, G.},
  \bibinfo{author}{Mininni, R.~M.}, \& \bibinfo{author}{Romanelli, S.}
  (\bibinfo{year}{2015}).
\newblock \bibinfo{title}{Identification of a diffusion coefficient in strongly
  degenerate parabolic equations with interior degeneracy}.
\newblock {\it \bibinfo{journal}{Journal of Evolution Equations}\/},  (pp.
  \bibinfo{pages}{27--51}).
%Type = Article
\bibitem[{Fragnelli \& Mugnai(2013)}]{fragnelli2013carleman}
\bibinfo{author}{Fragnelli, G.}, \& \bibinfo{author}{Mugnai, D.}
  (\bibinfo{year}{2013}).
\newblock \bibinfo{title}{Carleman estimates and observability inequalities for
  parabolic equations with interior degeneracy}.
\newblock {\it \bibinfo{journal}{Advances in Nonlinear Analysis}\/},  {\it
  \bibinfo{volume}{2}\/}, \bibinfo{pages}{339--378}.

\end{thebibliography}
	
\end{document}